\documentclass[12pt,reqno]{amsart}

\usepackage{amssymb, xcolor}
\usepackage{amscd}
\usepackage{amsfonts}
\usepackage{setspace}
\usepackage{version}
\usepackage{tikz}
\usepackage{caption}
\usepackage{graphicx}

\newtheorem{theorem}{Theorem}[section]
\newtheorem{lemma}[theorem]{Lemma}
\newtheorem{proposition}[theorem]{Proposition}
\newtheorem{corollary}{Corollary}

\usepackage{color}
\usepackage{amssymb}

\theoremstyle{definition}

\renewcommand{\leq}{\leqslant}
\renewcommand{\geq}{\geqslant}

\def\R{\mathbb{R}}

\def\Z{\mathbb{Z}}

\def\N{\mathbb{N}}

\numberwithin{equation}{section}

\begin{document}

\title[On suprema of convolutions on discrete cubes]{
On suprema of convolutions on discrete cubes 
}





\author[J. Gaitan]{Jos\'e Gaitan}
\address[JG]{Department of Mathematics, Virginia Polytechnic Institute and State University,  225 Stanger Street, Blacksburg, VA 24061-1026, USA
}
\email{jogaitan@vt.edu}

\author[J. Madrid]{Jos\'e Madrid}
\address[JM]{Department of Mathematics, Virginia Polytechnic Institute and State University,  225 Stanger Street, Blacksburg, VA 24061-1026, USA
}
\email{josemadrid@vt.edu}


\maketitle
\begin{abstract}

We find the optimal constant $C$ such that 
\begin{equation*}
\|f_1*f_2*\dots*f_{k}\|_{\infty}\geq C\prod_{i=1}^{k}\|f_i\|_1
\end{equation*}
for functions $f_i:\{0,1\}^d\to\R$. As applications, we derive bounds for Sidon sets on hypercubes, and, we also obtain bounds for the continuous analogue problem.


\end{abstract}

\section{Introduction}
For subsets $A_1, A_2, \dots A_n$ of an additive group $G$, we define the sumset
$$
A_1+A_2\dots+A_n:=\{a_1+a_2+\dots+a_n;\ a_i\in A_i\ \text{for all}\ i\},
$$
and $kA:=A_1+A_2+\dots+A_k$, $A_i=A$ for $1\leq i\leq k$.

A set $A\subset G$ is called $g-$Sidon set of order $k$, if
\begin{equation}
|\{(a_1,a_2,\dots,a_k)\in A^k; a_1+a_2+\dots+a_k=a\}|\leq g
\end{equation}
for all $a\in kA$.

Since the results by Cilleruelo, Ruzsa \& Vinuesa \cite{CRV} relating optimal bounds for the size of Sidon sets to optimal bounds for suprema of convolutions.
These problems have attracted the attention of many authors. Specifically, the problem of finding the best constant $C$ such that
$$
\|f*f\|_{\infty}\geq C\|f\|^2_{1}
$$
for all nonnegative functions $f:\R\to\R$ supported on $(-1/4,1/4)$, has been studied in \cite{CRT}, \cite{CS}, \cite{G},\cite{MO1},\cite{MV} and \cite{Y}. The current best known bounds are
\begin{equation}
1.28<C<1.51.
\end{equation}
The lower bound was established by Cloninger and Steinerberger \cite{CS}, the upper bound was previously observed by Matolcsi \& Vinuesa in \cite{MV} using numerical analysis for step functions. It was also observed by Matolcsi \& Vinuesa \cite[Note 4.1]{MV} that 
$$
C=\lim_{m\to\infty} 2(m+1)\bar C_{2,m}, \ \text{and}\ \ C\leq 2(m+1)\bar C_{2,m} \ \text{for all}\ m\geq 1,
$$
where, for any $k,m\in\N$, $\bar C_{k,m}$ denotes the best constant such that
\begin{equation*}
\|\underbrace{f*f*\dots*f}_{k}\|_{\infty}\geq\bar C_{k,m}\|f\|^k_1
\end{equation*}
for all $f:\{0,1,\dots,m\}\to\R$. We also denote by $C_{k,m}$ the best constant such that
\begin{equation*}
\|{f_1*f_2*\dots*f_k}\|_{\infty}\geq\ C_{k,m}\prod_{i=1}^k\|f_i\|_1
\end{equation*}
for all functions $f_i:\{0,1,\dots,m\}\to\R$, $1\leq i\leq k$.



\section{Main results and convolution inequalities on $\{0,1\}^d$}

Our main result gives a precise formula for $C_{k,1}$.

\begin{theorem}\label{thm: main theorem}
Let $k\in\N, k\geq 2$, then
\begin{align*}
C_{k,1}=\bar C_{k,1}&={{k}\choose{\lfloor k/2 \rfloor}} \left(\frac{\lfloor \frac{k+1}{2}\rfloor \lceil \frac{k+1}{2}\rceil}{(\lfloor \frac{k+1}{2}\rfloor +\lceil \frac{k+1}{2}\rceil)^2}\right)^{\frac{k}{2}}\\
&=\begin{cases}
\frac{{{k}\choose{\lfloor k/2 \rfloor}}}{2^{k}}, & \text{if}\ k\ \text{odd},\\
\frac{{{k}\choose{\lfloor k/2 \rfloor}}}{2^k}\left(1-\frac{1}{(k+1)^2}\right)^{\frac{k}{2}},  & \text{if}\ k\ \text{even}.
\end{cases}
  \end{align*}
\end{theorem}

{\it{Remark: }}
 Let $k\in\N, k\geq 2$. Similarly to \cite[Note 4.1]{MV}. If we denote by $C_k$ the best constant such that
\begin{equation*}
\|\underbrace{f*f*\dots*f}_{k}\|_{\infty}\geq C_{k}\|f\|^k_1
\end{equation*}
holds for all nonnegative functions $f:\R\to\R$ supported on $(-\frac{1}{2k},\frac{1}{2k})$, then
$$
C_k=\lim_{m\to\infty} k(m+1)\bar C_{k,m}, \ \text{and}\ \ C_k\leq k(m+1)\bar C_{2,m} \ \text{for all}\ m\geq 1.
$$
This follows immediately by considering simple functions\\ $f(x):=\sum_{j=1}^{k}a_j\chi_{[-\frac{1}{2k}+\frac{j-1}{k^2},-\frac{1}{2k}+\frac{j}{k^2}]}(x)$ for $a_1,\dots,a_k\in\R_{\geq0}$.
\begin{corollary}
Let $k\in\N, k\geq 2$. The following inequality holds
\begin{equation*}
 C_{k}\leq 2kC_{k,1}.
\end{equation*} 
\end{corollary}

In recent years many convolution inequalities for real valued functions on the hypercube have been studied motivated by applications to additive combinatorics and information theory. For instance, to obtain bouds for additive energies \cite{KT}, \cite{DGIM}, number of disjoint partitions \cite{I}, and sumsets \cite{BIKM}, \cite{MRSZ}, \cite{GMRSZ}. 

For any $k,d \geq 1$, and any $f_{1}, \ldots, f_{k} :\{0,1\}^{d} \to \mathbb{R}$ we have that $f_{1}*f_{2}*\ldots*f_{k}$ is supported on $\{0,1,\dots,k\}^d$, then (since the maximum of a collection of numbers is greater than the average)
\begin{align}\label{trivial inequality}
\|f_{1}*f_{2}*\ldots*f_{k}\|_{\ell^{\infty}(\Z^d)}\geq \frac{1}{(k+1)^d}\prod_{j=1}^{k} \|f_{j}\|_{1}.
\end{align}

Our next result gives an optimal bound improving \eqref{trivial inequality}.

\begin{theorem}\label{thm: main theorem for d>1}
For any $k,d \geq 1$, and any $f_{1}, \ldots, f_{k} :\{0,1\}^{d} \to \mathbb{R}$ we have 
\begin{align}\label{maininequality}
\|f_{1}*f_{2}*\ldots*f_{k}\|_{\ell^{\infty}(\Z^d)}&\geq {{k}\choose{\lfloor k/2 \rfloor}}^d \left(\frac{\lfloor \frac{k+1}{2}\rfloor \lceil \frac{k+1}{2}\rceil}{(\lfloor \frac{k+1}{2}\rfloor +\lceil \frac{k+1}{2}\rceil)^2}\right)^{\frac{kd}{2}}\prod_{j=1}^{k} \|f_{j}\|_{1}.\\
&\hspace{-0.6cm}=\begin{cases}
\dfrac{{{k}\choose{\lfloor k/2 \rfloor}}^d}{2^{kd}}\prod_{j=1}^{k} \|f_{j}\|_{1}, & \text{if}\ k\ \text{odd},\\
\dfrac{{{k}\choose{\lfloor k/2 \rfloor}}^d}{2^{kd}}\Big(1-\frac{1}{(k+1)^2}\Big)^{\frac{kd}{2}}\prod_{j=1}^{k} \|f_{j}\|_{1},  & \text{if}\ k\ \text{even}.
\end{cases}
\end{align}
Moreover,  for each fixed $k$ this inequality is sharp, the equality is attained when $f_i(x_1,\dots,x_d):=\left(k-\lfloor\frac{k}{2}\rfloor\right)^{\sum_{i=1}^{d}x_i}\left(\lfloor\frac{k}{2}\rfloor+1\right)^{d-\sum_{i=1}^{d}x_i}$ for all $x\in\{0,1\}^d$, $1\leq i\leq k$. In particular, if $k$ is odd, the equality is attained when $f_i(x):=\left(\frac{k+1}{2}\right)^{d}$ for all $x\in\{0,1\}^d$, $1\leq i\leq k$.
\end{theorem}

\begin{corollary}\label{corollary for sumsets}
    For all $A\subset\{0,1\}^d$ we have that
    \begin{align}\label{ineq for sumsets}
        &\max_{x\in kA}|\{(a_1,\dots,a_k)\in A^k; a_1+\dots+a_k=x\}|\nonumber\\
        &\geq {{k}\choose{\lfloor k/2 \rfloor}}^d \left(\frac{\lfloor \frac{k+1}{2}\rfloor \lceil \frac{k+1}{2}\rceil}{(\lfloor \frac{k+1}{2}\rfloor +\lceil \frac{k+1}{2}\rceil)^2}\right)^{\frac{kd}{2}}|A|^k.
    \end{align}
    In particular, if $k$ is odd and $A$ is a $g-$Sidon set of order $k$ then
    $|A|^k\leq g\frac{2^{kd}}{{{k}\choose{\lfloor k/2 \rfloor}}^d}$.
\end{corollary}

\noindent
{\it{Remark:}} If $k$ is odd and $A=\{0,1\}^d$ then both sides of \eqref{ineq for sumsets} are equal to ${{k}\choose{\lfloor k/2 \rfloor}}^d$.

\begin{proof}[Proof of Corollary \ref{corollary for sumsets}]
Follows from Theorem \ref{thm: main theorem for d>1} choosing $f_i=\chi_{A}$ for all $1\leq i\leq k$.
\end{proof}

\section{Proof of main results}

\begin{proof}[Proof of Theorem \ref{thm: main theorem for d>1}]
This follows from Theorem \ref{thm: main theorem} after a standard compressing dimension (tensorization) argument \cite[Proposition 2.1]{BIKM}, we include the details for completeness.

Assume that \eqref{maininequality} holds for $d=1$. For any $(\bar x,x')\in \Z^{d-1}\times\Z$, defining
$\bar f_i:\Z^{d-1}\to\R$ by $\bar f_i(z):=\sum_{y=0}^{1}f_i(\bar z,y)$, we have
\begin{align*}
&\max_{x\in\{0,1,\dots,k\}^d} f_1*f_2*\dots*f_k(\bar x,x')\\
&=
    \max_{\bar x\in\{0,1,\dots,k\}^{d-1}}\sum_{\sum \bar x_i=\bar x}\max_{x'\in\{0,1,\dots,k\}}\sum_{\sum x'_i=x'}\prod_{i=1}^{k}f_i(\bar x_i,x'_i)\\
    &\stackrel{\text{case}\, d=1}{\geq}  \max_{\bar x\in\{0,1,\dots,k\}^{d-1}}\sum_{\sum \bar x_i=\bar x}C\prod_{i=1}^{k}\sum_{y=0}^{1}f_i(\bar x_i,y)\\
    &=C\max_{\bar x\in\{0,1,\dots,k\}^{d-1}}\sum_{\sum \bar x_i=\bar x}\prod_{i=1}^{k}\bar f_i(\bar x_i)\\
    &=C\max_{x\in\{0,1,\dots,k\}^{d-1}} \bar f_1*\bar f_2*\dots*\bar f_k(\bar x)\\
    &\stackrel{\text{iterate}}{\geq}\ldots C^{d}\prod_{i=1}^{k}\|f_i\|_1.
\end{align*}
\end{proof}

\begin{proof}[Proof of Theorem \ref{thm: main theorem}]
Given functions $f_{i}:\{0,1\}\to\R$ for $1\leq i\leq k$. Without loss of generality, we assume that $f_i\neq0$, otherwise \eqref{maininequality} holds trivially. Defining $x_{i}:=\frac{f_{i}(1)}{f_{i}(0)+f_{i}(1)}$ for each $1\leq i\leq k$, \eqref{maininequality} is equivalent to prove 
\begin{align}\label{eq: x1 x2 x3 ...}
&\inf_{0\leq x_i\leq 1}\max_{0\leq m\leq k}\sum_{i_1,i_2,\dots,i_m}x_{i_1}\dots x_{i_m}\prod_{j\notin\{i_1,\dots,i_m\}} (1-x_j)\\
&= {{k}\choose{\lfloor k/2 \rfloor}}^{} \left(\frac{\lfloor \frac{k+1}{2}\rfloor \lceil \frac{k+1}{2}\rceil}{(\lfloor \frac{k+1}{2}\rfloor +\lceil \frac{k+1}{2}\rceil)^2}\right)^{\frac{k}{2}}\nonumber.
\end{align}

\subsection{The diagonal case.}
The following Lemma correspond to the case $x_1=x_2=\dots=x_k$ in \eqref{eq: x1 x2 x3 ...}.
\begin{lemma}\label{Key lemma todas iguales} Let $k\geq 1$, the following identity holds
\begin{equation}
\bar C_{k,1}=\inf_{x\in[0,1]}\max_{0\leq m\leq k} {{k}\choose{m}}x^{k-m}(1-x)^m={{k}\choose{\lfloor k/2 \rfloor}}^{} \left(\frac{\lfloor \frac{k+1}{2}\rfloor \lceil \frac{k+1}{2}\rceil}{(\lfloor \frac{k+1}{2}\rfloor +\lceil \frac{k+1}{2}\rceil)^2}\right)^{\frac{k}{2}}.
\end{equation}
\end{lemma}

\begin{proof}[Proof of Lemma \ref{Key lemma todas iguales}]
We observe that
\begin{align}\label{comparion de terminos}
{{k}\choose{i}}x^{k-i}(1-x)^i&={{k}\choose{i+1}}x^{k-i-1}(1-x)^{i+1}\frac{i+1}{k-i}\frac{x}{1-x}\nonumber\\
&\geq {{k}\choose{i+1}}x^{k-i-1}(1-x)^{i+1},
\end{align}
for all $x\geq \frac{k-i}{k+1}$ (with equality for $x=\frac{k-i}{k+1}$). Then, for any $0\leq i\leq k$
\begin{equation*}
    \max_{0\leq m\leq k} {{k}\choose{m}}x^{k-m}(1-x)^m={{k}\choose{i}}x^{k-i}(1-x)^i,
\end{equation*}
for all $x\in \left[\frac{k-i}{k+1},\frac{k+1-i}{k+1}\right]$.

For each $k\geq 1$ and $0\leq i\leq k$, we define the functions $g_{k,i}:\left[\frac{k-i}{k+1},\frac{k+1-i}{k+1}\right]\to\R$ by $g_{k,i}(x):=x^{k-i}(1-x)^i$. We observe that $g_{k,i}$ is increasing in $[\frac{k-i}{k+1},\frac{k-i}{k}]$ and decreasing in $[\frac{k-i}{k},\frac{k+1-i}{k+1}]$, since $g'_{k,i}(x)=x^{k-i-1}(1-x)^{i-1}[(k-i)-kx]$.
Therefore
$$
g_{k,i}(x)\geq \min\left\{g_{k,i}\left(\frac{k-i}{k+1}\right),g_{k,i}\left(\frac{k+1-i}{k+1}\right)\right\}
$$
for all $x\in \left[\frac{k-i}{k+1},\frac{k+1-i}{k+1}\right]$. Moreover, observe that by \eqref{comparion de terminos} we have that
${{k}\choose{i}}g_{k,i}(\frac{k-i}{k+1})={{k}\choose{i+1}}g_{k,i+1}(\frac{k+1-(i+1)}{k+1})$. Then, since the function $h:[0,+\infty)\to\R$ defined by
$h(x):=\left(\frac{x}{x+1}\right)^x$, is a decreasing function (decrease to $\frac{1}{e}$ as $x\to\infty$), we have that
$$
\left(\frac{k-i}{k-i+1}\right)^{k-i}\leq \left(\frac{i}{i+1}\right)^{i}
$$
for all $i\leq \frac{k}{2}$. Equivalently, we have
$$
g_{k,i}\left(\frac{k-i}{k+1}\right)\leq g_{k,i}\left(\frac{k+1-i}{k+1}\right).
$$
From this, by symmetry, we conclude that
\begin{align*}
    \inf_{x\in[0,1]}\max_{0\leq m\leq k} {{k}\choose{m}}x^{k-m}(1-x)^m&=\min_{0\leq i\leq k/2} {{k}\choose{i}}g_{k,i}\left(\frac{k-i}{k+1}\right)\\
    &={{k}\choose{\lfloor \frac{k}{2}\rfloor}}g_{k,\lfloor \frac{k}{2}\rfloor}\left(\frac{k-\lfloor \frac{k}{2}\rfloor}{k+1}\right)\\
    &={{k}\choose{\lfloor k/2 \rfloor}}^d \left(\frac{\lfloor \frac{k+1}{2}\rfloor \lceil \frac{k+1}{2}\rceil}{(\lfloor \frac{k+1}{2}\rfloor +\lceil \frac{k+1}{2}\rceil)^2}\right)^{\frac{k}{2}}.
\end{align*}
\end{proof}

\subsection{From the general case to the diagonal case.}
By Lemma \ref{Key lemma todas iguales} it is enough to prove that $C_{k,1}=\bar C_{k,1}$, for this, a probabilistic interpretation will be convenient. Let $\boldsymbol{p}=(p_1,\dots,p_{k})\in(0,1)^{k}$ to be a $k$-tuple of parameters for $k\geq 2$.  For all $1\leq i\leq k$, observe that each function $f_i:\{0,1\}\to[0,1]$ defined by $f_i(1)=p_i$ and $f_i(0)=1-p_i$ is a Bernoulli random variable with success probability $p_i\in[0,1]$. The convolution $f_{1}*f_{2}*\ldots*f_{k}:\{0,1,\dots,k\}\to\R$  can be represented by $$f_{k,i}(\boldsymbol{p}):=(f_{1}*f_{2}*\ldots*f_{k})(i).$$
 The value of $f_{k,i}(\boldsymbol{p})$ equals the probability of having $i$ successful trials on a collection of $k$ independent Bernoulli random variables, this is the definition of the probability mass function of the \textit{Poisson Binomial distribution} (PB pmf) with success probabilities $[p_1,\dots , p_k]$ (also often called \textit{Bernoulli sum}).

 \subsubsection{Proof strategy, notation and basic properties.}
 We will interpret the convolutions $f_{k,i}(\boldsymbol{p})$ using this Probability theory viewpoint, and we will make use of the Poisson Binomial distribution properties. In \ref{intersection points} we reduce the analysis to intersection points $\bigcup_{i=1}^k \mathcal{P}_{k,i},$ where $\mathcal{P}_{k,i}:=\{\boldsymbol{p}\in(0,1)^k: f_{k,i}(\boldsymbol{p})=f_{k,i-1}(\boldsymbol{p})\}$.
In \ref{Lagrange multipliers} we prove that the infimum value of
 \begin{align}\label{problem linf}
     \inf_{\boldsymbol{p}\in[0,1]^{k}} \|f_{1}*f_{2}*\ldots*f_{k}\|_\infty=\inf_{\boldsymbol{p}\in[0,1]^{k}} \max_{0\leq i\leq k}\{f_{k,i}(\boldsymbol{p}) \}
\end{align}
must be attained when $f_1=f_2=\ldots=f_k$, which reduces the problem to the Lemma \ref{Key lemma todas iguales}.

We denote by  $\boldsymbol{p}_j':=(p_1,p_2,\dots,p_{j-1},p_{j+1},\dots,p_k)\in(0,1)^{k-1}$, the parameter vector $\boldsymbol{p}$ with the $j$-th entry removed. Then $$f_{k-1,i}(\boldsymbol{p}_j'):=(f_{1}*f_{2}*\ldots*f_{j-1}*f_{j+1}*\dots*f_{k})(i),$$
is the probability of $i$ successful trials with the $j-$th Bernoulli random variable ignored. We have the recursive relation 
 \begin{equation}\label{recursion}
f_{k,i}(\boldsymbol{p})= (1-p_j)f_{k-1,i}(\boldsymbol{p}_j')+(p_j)f_{k-1,i-1}(\boldsymbol{p}_j'),
 \end{equation}
which is valid for all $0\leq i,j\leq k$, with the convention that $f_{k-1,k}(\boldsymbol{p}_j')=f_{1,-1}(\boldsymbol{p}_j')=0$.\\

The first property concerning Poisson Binomial distributions is that they are \textit{unimodal}. This is, ``first increasing, then decreasing, and the mode is either unique or shared by two adjacent integers $i$ and $i-1$'' (As described by \cite{WANG}). More precisely, for each choice of parameters $\boldsymbol{p}\in[0,1]^k$ there is a unique index $0\leq i\leq k$ such that \begin{equation}\label{unimodality}
f_{k,k}(\boldsymbol{p})<\dots <f_{k,i+1}(\boldsymbol{p})< f_{k,i}(\boldsymbol{p})\geq f_{k,i-1}(\boldsymbol{p})>\dots>f_{k,1}(\boldsymbol{p})>f_{k,0}(\boldsymbol{p}).  
\end{equation}
 where this index value, $i$, is called \textit{the mode}. For each $\boldsymbol{p}\in[0,1]^k$, the maximum of the sequence $\{f_{k,j}(\boldsymbol{p})\}$ is achieved uniquely at $f_{k,i}(\boldsymbol{p})$ or it can be shared by $f_{k,i}(\boldsymbol{p})=f_{k,i-1}(\boldsymbol{p})$, this is called \textit{unimodality}.. 
 
It is also well known that the Poisson Binomial is a \textit{log-concave} distribution, this is 
\begin{equation}\label{log-concave}
    f_{k,i}(\boldsymbol{p})^2\geq f_{k,i-1}(\boldsymbol{p})f_{k,i+1}(\boldsymbol{p})
\end{equation}
for $1\leq i\leq k$, and for all $\boldsymbol{p}\in[0,1]^k$. But it indeed belongs to a class of distributions satisfying an even stronger notion, the \textit{ultra log-concavity}, often defined as 
\begin{equation}\label{ultra log-concave}
    f_{k,i}(\boldsymbol{p})^2\geq \left(\dfrac{i+1}{i}\right)\left(\dfrac{k-i+1}{k-i}\right) f_{k,i-1}(\boldsymbol{p})f_{k,i+1}(\boldsymbol{p}).
\end{equation}
for $1\leq i\leq k-1$, see \cite[Theorem 2]{WANG} or \cite[Corollary 4.2]{TANG}.

\subsubsection{Reduction to intersection points}\label{intersection points}

It will be useful to think about $f_{k,i}(\boldsymbol{p})=f_{k,i}(\boldsymbol{p}_j',p_j)$ as a linear function $f_{k,i}(\boldsymbol{p}_j',\ \cdot):[0,1] \rightarrow [0,1]$, defined by (\ref{recursion}). One more property of the probability distributions is the following monotone likelihood ratio.

\begin{proposition}\label{ratios are increasing}
Let $1\leq i\leq k$ and $\boldsymbol{p}'_j\in (0,1)^{k-1}$. Then
\begin{enumerate} 
    \item \begin{equation}\label{ratios f}
        r_{k,i}(\boldsymbol{p}'_j,p_j):=\dfrac{f_{k,i}(\boldsymbol{p}'_j,p_j)}{f_{k,i-1}(\boldsymbol{p}'_j,p_j)}
    \end{equation}
    is an increasing function on $p_j\in[0,1]$.
\item  \begin{equation}\label{ratios f monotonicity on index}
        r_{k,i+1}(\boldsymbol{p})<  r_{k,i}(\boldsymbol{p})
    \end{equation}
   decreases on the index $i$ for all $1\leq i\leq k$, and any $\boldsymbol{p}\in[0,1]^k$.
   \end{enumerate}
\end{proposition}
{\it{Remark:}} Indeed, the ratios $r_{k,i}(\boldsymbol{p}'_j,p_j)$ are also concave on $p_j\in[0,1]$, but we only need monotonicity for our purposes.
\begin{proof}
By the recursive relation (\ref{recursion}), both the numerator and denominator of the expression of $r_{k,i}$ are linear functions on $p_j$, more precisely
\begin{equation}\label{ratios expanded}
    r_{k,i}(\boldsymbol{p}'_j,p_j) =\dfrac{p_j\Big({f_{k-1,i-1}(\boldsymbol{p}_{j}')-f_{k-1,i}(\boldsymbol{p}_{j}')}\Big)+f_{k-1,i}(\boldsymbol{p}_{j}')}{p_j\Big(f_{k-1,i-2}(\boldsymbol{p}_{j}')-f_{k-1,i-1}(\boldsymbol{p}_{j}')\Big)+f_{k-1,i-1}(\boldsymbol{p}_{j}')}.
\end{equation}
Then, the sign of the derivative $\dfrac{dr_{k,i}}{dp_j}$ is given by 
\begin{align*}
    &\big({f_{k-1,i-1}(\boldsymbol{p}_{j}')-f_{k-1,i}(\boldsymbol{p}_{j}')}\big)f_{k-1,i-1}(\boldsymbol{p}_{j}')\\
    &-\big(f_{k-1,i-2}(\boldsymbol{p}_{j}')-f_{k-1,i-1}(\boldsymbol{p}_{j}')\big)f_{k-1,i}(\boldsymbol{p}_{j}')\\
    =&f_{k-1,i-1}(\boldsymbol{p}_{j}')^2-f_{k-1,i}(\boldsymbol{p}_{j}')f_{k-1,i-2}(\boldsymbol{p}_{j}')>0
\end{align*}
where the last inequality follows from the ultra log-concavity property (\ref{ultra log-concave}), then $\dfrac{dr_{k,i}}{dp_j}>0$ for all $p_j\in(0,1)$.

The second part follows by (\ref{ultra log-concave}) as well, since $\left(\frac{i+1}{i}\right)\left(\frac{k-i+1}{k-i}\right)>1$, then, rearranging the terms we get $\frac{f_{k,i}(\boldsymbol{p})}{f_{k,i-1}(\boldsymbol{p})}>\frac{f_{k,i+1}(\boldsymbol{p})}{f_{k,i}(\boldsymbol{p})}$, or equivalently $r_{k,i}(\boldsymbol{p})> r_{k,i+1}(\boldsymbol{p})$.
\end{proof}
By the previous proposition, for a fixed $\boldsymbol{p}_{j}'\in(0,1)^{d-1}$, the ratio $r_{k,i}(\boldsymbol{p}'_j,p_j)$ equals $1$ in at most one value of $p_j\in[0,1]$. Indeed, by solving for $p_j$ in (\ref{ratios expanded}), we obtain that $r_{k,i}(\boldsymbol{p}'_j,p_j^*)=1$ if and only if $$p_j^*:=\dfrac{f_{k-1,i-1}(\boldsymbol{p}_{j}')-f_{k-1,i}(\boldsymbol{p}_{j}')}{2f_{k-1,i-1}(\boldsymbol{p}_{j}')-f_{k-1,i}(\boldsymbol{p}_{j}')-f_{k-1,i-2}(\boldsymbol{p}_{j}')}\in[0,1].$$

By monotonicity of $r_{k,i}(\boldsymbol{p}'_j,p_j)$ we have the following result about the preservation of the leading mode for values of $p_j$ around $p_j^*$.
\begin{proposition} Let a choice of parameters $[\boldsymbol{p}'_j,p_j^*]\in[0,1]^k$, where the index $i$, is the leading mode, then
\begin{equation*}\label{maximum function around p_j^*}
\max_{0\leq i\leq k}\{f_{k,i}(\boldsymbol{p}'_j,p_j) \}=\begin{cases} 
      f_{k,i-1}(\boldsymbol{p}'_j,p_j), & \text{if } \ \  p_j\in(0,p_j^*)\\
     f_{k,i-1}(\boldsymbol{p}'_j,p_j)=f_{k,i}(\boldsymbol{p}'_j,p_j), & \text{if } \ \  p_j=p_j^* \\
     f_{k,i}(\boldsymbol{p}'_j,p_j), & \text{if } \ \  p_j\in(p_j^*,1) \\
   \end{cases}
\end{equation*}
Moreover, the following Figure \ref{fig:min of max} is accurate for all $p_j\in(0,1).$
\begin{figure}[h]
\centering
\begin{tikzpicture}[scale=4]
    \draw[->] (-0.15,0) -- (1.2,0) node[right] {\small $p_j$}; 
    \draw[->] (-0,-.1) -- (-0,1.3) ;
    \draw[-] (0,0.6)  -- (1,0.95) node[right,black] {$\textcolor{black}{f_{k,i}}$};
    \draw[red] (0,0.15)  -- (1,0.6) node[right,black] {$\textcolor{red}{f_{k,i+1}}$};
    \draw[-] (0,1)  -- (1,0.4) node at (-0.15,1.05) {$\textcolor{black}{f_{k,i-1}}$};
    \draw[blue] (0,0.4)  -- (1,0.07) node at (-0.15,0.45) {$\textcolor{blue}{f_{k,i-2}}$};
    
    \fill (0.421,0.7473)  circle[radius=0.4pt];
    \fill[black] (0,1)  circle[radius=0.4pt];
    \fill[blue] (0,0.4)  circle[radius=0.4pt];
     \fill[red] (0,0.6)  circle[radius=0.4pt];
    \fill[black] (1,0.95)  circle[radius=0.4pt];
    \fill[red] (1,0.6)  circle[radius=0.4pt];
    \fill[blue] (1,0.4)  circle[radius=0.4pt];
    \draw[loosely dashed](0.421,0) -- (0.421,0.7473) node at (0.421,-0.1) {\footnotesize$p_j^*$};
    \draw[-](1,-0.02) -- (1,0.02) node at (1,-0.07) {\footnotesize$1$};
    \node at (-0.05,-0.05) {\footnotesize$0$};
\end{tikzpicture}
\caption{}
\label{fig:min of max}
\end{figure} 
\end{proposition}
\begin{proof}
We will verify that $f_{k,i}$ is a linear equation on $p_j$ with positive slope, whereas $f_{k,i-1}$ has negative slope. This can be seen by using (\ref{ratios expanded}) and the monotonicity of $r_{k,i}$, note that $$\dfrac{f_{k-1,i}(\boldsymbol{p}_{j}')}{f_{k-1,i-1}(\boldsymbol{p}_{j}')}=r_{k,i}(\boldsymbol{p}_{j}',0)<r_{k,i}(\boldsymbol{p}_{j}',p_j^*)=1,$$ 
thus, $f_{k-1,i}(\boldsymbol{p}_{j}')<f_{k-1,i-1}(\boldsymbol{p}_{j}')$, then the slope of $f_{k,i}$ is positive, by \eqref{recursion}. Similarly, using that $r_{k,i}(\boldsymbol{p}_{j}',1)>r_{k,i}(\boldsymbol{p}_{j}',p_j^*)=1$, we conclude that the slope of $f_{k,i-1}$ is negative. 

To prove that $f_{k,i+1}(\boldsymbol{p}_{j}',p_j)<f_{k,i}(\boldsymbol{p}_{j}',p_j)$ for all $p_j\in (0,1)$ we observe that $f_{k,i+1}(\boldsymbol{p}_{j}',1)=f_{k-1,i}(\boldsymbol{p}_{j}')=f_{k,i}(\boldsymbol{p}_{j}',0)$ and the slope of $f_{k,i+1}$ as a linear function of $p_j$ is positive as shown in Figure \ref{fig:min of max}, since
$$\dfrac{f_{k-1,i+1}(\boldsymbol{p}_{j}')}{f_{k-1,i}(\boldsymbol{p}_{j}')}=r_{k,i+1}(\boldsymbol{p}_{j}',0)<r_{k,i}(\boldsymbol{p}_{j}',0)<r_{k,i}(\boldsymbol{p}_{j}',p_j^*)=1,$$ 
by Proposition \ref{ratios are increasing}.

Similarly, we have that $f_{k,i-1}(\boldsymbol{p}_{j}',p_j)>f_{k,i-2}(\boldsymbol{p}_{j}',p_j)$ for all $p_j\in (0,1)$  since $f_{k,i-2}(\boldsymbol{p}_{j}',0)=f_{k-1,i-2}(\boldsymbol{p}_{j}')=f_{k,i-1}(\boldsymbol{p}_{j}',1)$ and the slope of $f_{k,i-2}$ is negative since $r_{k,i-1}(\boldsymbol{p}_{j}',1)>r_{k,i}(\boldsymbol{p}_{j}',1)>r_{k,i}(\boldsymbol{p}_{j}',p_j^*)=1$. The same argument works for the remaining $\{f_{k,j}\}_{j=i+2}^k$ and $\{f_{k,j}\}_{j=0}^{i-3}$.
\end{proof}

As described in Figure \ref{fig:min of max}, this means that the intersection points $p_j^*$ are the locations of all the local minimum values of the function from Proposition \ref{maximum function around p_j^*}.
\begin{proposition}\label{minimizer is on the intersections}
The value $\displaystyle \inf_{\boldsymbol{p}\in[0,1]^k} \max_{0\leq i\leq k}\{f_{k,i}(\boldsymbol{p})\}$ must occur at parameter values where the mode is shared, this is, where two functions intersect. So, the minimizer parameter vector lies in $\boldsymbol{p}\in\bigcup_{i=1}^k \mathcal{P}_{k,i},$ where $\mathcal{P}_{k,i}:=\{\boldsymbol{p}\in(0,1)^k: f_{k,i}(\boldsymbol{p})=f_{k,i-1}(\boldsymbol{p})\}$.
\end{proposition}

\subsubsection{Some auxiliary results}
The next theorem is split into two parts,
the first part was proved by Newton, see \cite[Section 2.22, Theorem 51]{HLP} or \cite[Section 4.3]{HLP} for an elementary proof, this establish the classic relation between a sequence of positive numbers $\{a_j\}$ satisfying a \textit{Newton's inequality} and real-rooted polynomials with real coefficients. The second part shows the relation of such inequalities with Poisson distributions, first observed by Aissen, Schoenberg, and Whitney in \cite{ASW}, see \cite[Theorem 4.1]{TANG}.
\begin{theorem}
\begin{enumerate}
    \item  Let $\{a_j\}_{j=0}^{k}$ be a finite sequence of real numbers such that the generating polynomial $\displaystyle P(z)=\sum_{j=0}^{k}a_jz^j$ has only real roots. Then
 \begin{equation}\label{Newton's inequality}
 \left(\dfrac{a_i}{\binom{k}{i}}\right)^2\geq \dfrac{a_{i-1}}{\binom{k}{i-1}}\dfrac{a_{i+1}}{\binom{k}{i+1}},    
 \end{equation}
This is, any real rooted polynomial with real coefficients only (not necessarily positive) satisfy the Newton's inequality (\ref{Newton's inequality}).\\
\item Additionally, if we impose the condition that the sequence $\{a_j\}_{j=0}^{k}$ is made of non-negative numbers. Then, the following are equivalent:
 \begin{itemize}
     \item The polynomial $\displaystyle P(z)=\sum_{j=0}^{k}a_jz^j$ has only real roots.\\
     
     \item The sequence $\{a_j\}_{j=0}^{k}$ is \textit{log-concave}, and unimodal.\\
     
      \item The sequence $\dfrac{a_1}{P(1)},\dots,\dfrac{a_k}{P(1)}$ is the probability distribution of a Poisson Binomial distribution with success probabilities $\boldsymbol{p}=(p_1,\dots,p_k)$. 
 \end{itemize}
 \end{enumerate}
\end{theorem}
We know that $\{f_{k,j}(\boldsymbol{p})\}_{j=0}^{k}$ are the probability distributions of a Poisson Binomial distribution, therefore the generating polynomial $\displaystyle P(z)=\sum_{j=0}^{k}f_{k,j}(\boldsymbol{p}) z^j$, on the complex variable $z$ is real rooted. A direct consequence of this is the fact that the sequence of successive differences $D_{k,j}=f_{k,j}-f_{k,j-1}$ for all $1\leq j\leq k$, satisfies the Newton's inequality (\ref{Newton's inequality}). This is because if we consider the generating polynomial \begin{align*}
  Q(z)&=\sum_{j=0}^{k} D_{k,j}(\boldsymbol{p})z^j\\
  &=\sum_{j=0}^{k} (f_{k,j}(\boldsymbol{p})-f_{k,j-1}(\boldsymbol{p}))z^j\\
  &=(1-z)\sum_{j=0}^{k} f_{k,j}(\boldsymbol{p})z^j.
\end{align*}
This shows that $Q(z)$ is also a real-rooted polynomial with real coefficients $\{D_{k,j}(\boldsymbol{p})\}_{j=0}^{k}$ obeying \begin{equation}\label{Newton's inequalities D}
\Big( D_{k,i}(\boldsymbol{p})\Big)^2\geq \left(\dfrac{i+1}{i}\right)\left(\dfrac{k-i+1}{k-i}\right) D_{k,i-1}(\boldsymbol{p})\cdot D_{k,i+1}(\boldsymbol{p}_{12}')
\end{equation}
for all $1\geq k$ and all $1\leq i\leq k-1$, as given by (\ref{Newton's inequality}). \\

Note that $\left(\frac{i+1}{i}\right)\left(\frac{k-i+1}{k-i}\right)>1$ for all $1\leq i\leq k-1$, so the Newton's inequalities $\Big( D_{k,i}(\boldsymbol{p})\Big)^2> D_{k,i-1}(\boldsymbol{p}) D_{k,i+1}(\boldsymbol{p}_{12}')$ are strict.

Lastly, note that the sequence $\{D_{k,j}(\boldsymbol{p})\}_{j=0}^{k}$ satisfies the same recursive relation (\ref{recursion}), this is \begin{equation}\label{recursion Ds}
D_{k,i}(\boldsymbol{p})= (1-p_j)D_{k-1,i}(\boldsymbol{p}_j')+(p_j)D_{k-1,i-1}(\boldsymbol{p}_j'),
\end{equation}
valid for all $1\leq i,j\leq k.$ 

\subsubsection{Lagrange Multipliers and conclusion of the argument.}\label{Lagrange multipliers}
By Proposition \ref{minimizer is on the intersections}, we have reduced our original problem, to solve a minimization problem on the restricted set of parameters $\bigcup_{i=1}^k \mathcal{P}_{k,i}.$
\begin{theorem}\label{Lagrange Multipliers}
Let $\mathcal{P}_{k,i}:=\{\boldsymbol{p}\in(0,1)^k: f_{k,i}(\boldsymbol{p})=f_{k,i-1}(\boldsymbol{p})\}$ to be the $(k-1)-$dimensional algebraic manifold of parameters that make $f_{k,i}$ to intersect $f_{k,i-1}$. 

The minimum value of $f_{k,i}(\boldsymbol{p})$ on $\mathcal{P}_{k,i}$ occurs when $p_1=p_2=\ldots=p_k$ for each $1\leq i \leq k$.  In other words, when $f_{k,i}(\boldsymbol{p}):=(f*f*\ldots*f)(i)$.
\end{theorem}
\begin{proof}
We use the Lagrange multiplier method, to minimize the values of $\displaystyle f_{k,i}(\boldsymbol{p})$, under the constraint set of parameters $\mathcal{P}_{k,i}:=\{\boldsymbol{p}\in(0,1)^k: f_{k,i}(\boldsymbol{p})=f_{k,i-1}(\boldsymbol{p})\}$. This is achieved at the minimum value of $f_{k,i}$ evaluated at all the solutions $(\boldsymbol{p},\lambda)\in \mathcal{P}_{k,i}\times \mathbb{R}$ satisfying \begin{equation}\label{Lagrange multiplier eqtn}
    \nabla f_{k,i}=\lambda \nabla(f_{k,i}-f_{k,i-1})
\end{equation}
where $\lambda\in \mathbb{R}$ is a constant and the gradient is taking over the parameters $(p_1,\dots,p_k).$
Each function $f_{k,i}(\boldsymbol{p})$ is differentiable on $p_{j}$ and its derivative is given by \begin{equation}\label{derivative on pj}
        \frac{\partial}{\partial p_{j}}\big( f_{k,i}(\boldsymbol{p})\big)=f_{k-1,i-1}(\boldsymbol{p}_j')-f_{k-1,i}(\boldsymbol{p}_j')=-D_{k-1,i}(\boldsymbol{p}_j'),
    \end{equation}
this follows immediately from (\ref{recursion}). Then
\begin{align*}
  \frac{\partial}{\partial p_{j}}\big( f_{k,i}(\boldsymbol{p})-f_{k,i-1}(\boldsymbol{p})\big)&= -D_{k-1,i}(\boldsymbol{p}_j')+D_{k-1,i-1}(\boldsymbol{p}_j')
\end{align*}
Since (\ref{Lagrange multiplier eqtn}) can be written as the system of equations
\begin{equation}\label{eq: equiv system}
\frac{\partial}{\partial p_{j}}\big( f_{k,i}(\boldsymbol{p})\big)=\lambda\ \frac{\partial}{\partial p_{j}}\big( f_{k,i}(\boldsymbol{p})-f_{k,i-1}(\boldsymbol{p})\big),
\end{equation}
for $1\leq j\leq k$. By the previous calculations, this is equivalent to
 \begin{equation}\label{lambda}
\frac{1}{\lambda}=\dfrac{D_{k-1,i}(\boldsymbol{p}_j')-D_{k-1,i-1}(\boldsymbol{p}_j')}{D_{k-1,i}(\boldsymbol{p}_j')},
\end{equation}
for all $1\leq j\leq k$. A solution to the Lagrange multiplier problem exists only if $\lambda$ remains constant independently of the $p_j$ that is been excluded. Our goal here is to prove that this can only happen when $p_1=p_2\dots=p_k$. We start by observing that $\lambda\neq 0$, otherwise, by \eqref{eq: equiv system} and \eqref{derivative on pj}, we have $D_{k-1,i}(\boldsymbol{p}_j')=0$ for any $1\leq j\leq k$, then for $\boldsymbol{p}\in\mathcal{P}_{k,i}$ we get
\begin{align*}
    0&=f_{k,i}(\boldsymbol{p})
-f_{k,i-1}(\boldsymbol{p})=(1-p_j)D_{k-1,i}(\boldsymbol{p}_j')+(p_j)D_{k-1,i-1}(\boldsymbol{p}_j')
\end{align*}
so $D_{k-1,i-1}(\boldsymbol{p}_j')=0$ as well, this implies that $f_{k-1,i-2}(\boldsymbol{p}_j')=f_{k-1,i-1}(\boldsymbol{p}_j')=f_{k-1,i}(\boldsymbol{p}_j')$, and this is not possible because the sequence is unimodal, meaning that at most two of the functions can be equal at the same time. From  \eqref{lambda} we obtain 
\begin{equation}\label{ratios lambda}
\dfrac{\lambda-1}{\lambda} :=\dfrac{D_{k-1,i-1}(\boldsymbol{p}_j')}{D_{k-1,i}(\boldsymbol{p}_j')},
\end{equation}
for all $1\leq j\leq k$. We claim that, if $p_{j_1}\neq p_{j_2}$, then $\dfrac{D_{k-1,i-1}(\boldsymbol{p}_{j_1}')}{D_{k-1,i}(\boldsymbol{p}_{j_1}')}\neq \dfrac{D_{k-1,i-1}(\boldsymbol{p}_{j_2}')}{D_{k-1,i}(\boldsymbol{p}_{j_2}')}$.
Assume that $p_1\neq p_2$, then $\boldsymbol{p}_{1}'\neq\boldsymbol{p}_2'$. We denote by $\boldsymbol{p}_{12}'$ the vector obtain from $\boldsymbol{p}$ after removing the coordinates $\{p_1,p_2\}$. We note that the expression (\ref{ratios lambda}) is a ratio of two consecutive functions just as (\ref{ratios f}), and moreover, the functions $D_{k,i}$ satisfy the recursive relation (\ref{recursion Ds}) which is the same that $f_{k,i}$ satisfies, thus by expanding the recursions we get
$$\dfrac{\lambda-1}{\lambda}=\dfrac{(p_2)\Big({D_{k-2,i-2}(\boldsymbol{p}_{12}')-D_{k-2,i-1}(\boldsymbol{p}_{12}')}\Big)+D_{k-2,i-1}(\boldsymbol{p}_{12}')}{(p_2)\Big(D_{k-2,i-1}(\boldsymbol{p}_{12}')-D_{k-2,i}(\boldsymbol{p}_{12}')\Big)+D_{k-2,i}(\boldsymbol{p}_{12}')},$$
where we have a fraction of two linear terms on $p_2$ and every difference $D_{k-2,j}(\boldsymbol{p}_{12}')$ is independent of $p_2$. We can repeat this process for $\boldsymbol{p}'_2$ and obtain analogous expressions by collecting the linear coefficients of $p_1$ $$\dfrac{\lambda-1}{\lambda}=\dfrac{(p_1)\Big({D_{k-2,i-2}(\boldsymbol{p}_{12}')-D_{k-2,i-1}(\boldsymbol{p}_{12}')}\Big)+D_{k-2,i-1}(\boldsymbol{p}_{12}')}{(p_1)\Big(D_{k-2,i-1}(\boldsymbol{p}_{12}')-D_{k-2,i}(\boldsymbol{p}_{12}')\Big)+D_{k-2,i}(\boldsymbol{p}_{12}')}.$$
This suggests the following definition 
 $$\Lambda(y):=\dfrac{y\Big({D_{k-2,i-2}(\boldsymbol{p}_{12}')-D_{k-2,i-1}(\boldsymbol{p}_{12}')}\Big)+D_{k-2,i-1}(\boldsymbol{p}_{12}')}{y\Big(D_{k-2,i-1}(\boldsymbol{p}_{12}')-D_{k-2,i}(\boldsymbol{p}_{12}')\Big)+D_{k-2,i}(\boldsymbol{p}_{12}')},$$
 for parameters $\boldsymbol{p}_{12}'$ fixed, this is a Mobius transformation on the variable $y$, meaning that it is either injective for all $y$ in its domain or identically constant if and only if $$\dfrac{D_{k-2,i-1}(\boldsymbol{p}_{12}')-D_{k-2,i}(\boldsymbol{p}_{12}')}{D_{k-2,i}(\boldsymbol{p}_{12}')}=\dfrac{D_{k-2,i-2}(\boldsymbol{p}_{12}')-D_{k-2,i-1}(\boldsymbol{p}_{12}')}{D_{k-2,i-1}(\boldsymbol{p}_{12})},$$
 equivalently
 $$\dfrac{D_{k-2,i-1}(\boldsymbol{p}_{12}')}{D_{k-2,i}(\boldsymbol{p}_{12}')}-1=\dfrac{D_{k-2,i-2}(\boldsymbol{p}_{12}')}{D_{k-2,i-1}(\boldsymbol{p}_{12})}-1.$$
or we can also write it as $$\Big( D_{k-2,i-1}(\boldsymbol{p}_{12}')\Big)^2=D_{k-2,i}(\boldsymbol{p}_{12}')\cdot D_{k-2,i-2}(\boldsymbol{p}_{12}').$$
But this equality is false by the strict Newton's inequality of the differences $\{D_{k-2,j}(\boldsymbol{p}_{12}'))\}_{j=1}^{k-2}$ as shown in (\ref{Newton's inequalities D}). This argument proves that $\Lambda(y)$ is a non-constant Mobius transformation, therefore injective on its domain, this concludes that $\Lambda(p_1)= \Lambda(p_2)=\dfrac{\lambda-1}{\lambda}$ if and only if $p_1=p_2$.

Recalling that we picked $p_1$ and $p_2$ arbitrarily, by iterating the same argument with all different choices for pairs, we obtain that the only solution for the Lagrange Multiplier problem is $p_1=p_2=\dots=p_k$ and $\lambda$ given by (\ref{lambda}), this solution gives the minimum value for $f_{k,i}(\boldsymbol{p})$ on the restricted set of parameters $\mathcal{P}_{k,i}$.
\end{proof}
Now we are all ready to finish the proof for Theorem \ref{thm: main theorem}. By the Proposition \ref{minimizer is on the intersections} the $\displaystyle\inf_{\boldsymbol{p}\in[0,1]^{k}} \max_{0\leq i\leq k}\{f_{k,i}(\boldsymbol{p}) \}$ is attained at $\boldsymbol{p}\in\bigcup_{i=1}^k \mathcal{P}_{k,i}$, and then by Theorem \ref{Lagrange Multipliers} the infimum must occur for some $\boldsymbol{p}$ that also satisfies $p_1=p_2=\dots=p_k$. Therefore,
\begin{align*}
 \displaystyle\inf_{\boldsymbol{p}\in[0,1]^{k}} \max_{0\leq i\leq k}\{f_{k,i}(\boldsymbol{p}) \}&=\displaystyle\inf_{\boldsymbol{p}\in\bigcup_{i}\mathcal{P}_{k,i}} \max_{0\leq i\leq k}\{f_{k,i}(\boldsymbol{p}) \}\\
 &=\displaystyle\inf_{p\in[0,1]} \max_{0\leq i\leq k}\{(f*f*\ldots*f)(i): f(1)=p \}   
\end{align*}
which is the reduction to Lemma \ref{Key lemma todas iguales}.
\end{proof}

\section{Acknowledgements.}
J.M. was partially supported by the AMS Stefan Bergman Fellowship and the Simons Foundation Grant $\# 453576$. J.M. is thankful to Zane Li for interesting discussion.


\end{document}